\documentclass[11pt,reqno]{article}
\usepackage{latexsym, amsmath, amssymb, a4, epsfig}
\usepackage{amsthm}
\usepackage{graphicx}
\usepackage{latexsym, color}

\newtheorem{thm}{Theorem}[section]
\newtheorem{cor}[thm]{Corollary}

\newtheorem{prop}[thm]{Proposition}
\newtheorem{rem}[thm]{Remark}

\usepackage[left=1in,right=1in,top=1in,bottom=1in]{geometry}

\newcommand{\SD}{\mathcal{S}_{\partial D}}
\newcommand{\SO}{\mathcal{S}_{\partial \Omega}}

\newcommand{\KDS}{\mathcal{K}_{\partial D}^{*}}

\newcommand{\KOS}{\mathcal{K}_{\partial \Omega}^{*}}

\newcommand{\bef}{\begin{figure}}
\newcommand{\enf}{\end{figure}}

\newcommand{\Sg}{\mathcal{S}_{\GG}}
\newcommand{\Dg}{\mathcal{D}_{\GG}}
\newcommand{\Kg}{\mathcal{K}_{\GG}}

\newcommand{\ds}{\displaystyle}

\newcommand{\p}{\partial}

\newcommand{\eqnref}[1]{(\ref {#1})}

\newcommand{\Rbb}{\mathbb{R}}

\newcommand{\la}{\langle}
\newcommand{\ra}{\rangle}

\newcommand{\Acal}{\mathcal{A}}

\newcommand{\Kcal}{\mathcal{K}}

\newcommand{\Scal}{\mathcal{S}}

\newcommand{\Ga}{\alpha}
\newcommand{\Gb}{\beta}

\newcommand{\Ge}{\epsilon}

\newcommand{\Gvf}{\varphi}
\newcommand{\Gg}{\gamma}

\newcommand{\Gl}{\lambda}

\newcommand{\Gm}{\mu}

\newcommand{\Gt}{\theta}

\newcommand{\Gr}{\rho}

\newcommand{\Gs}{\sigma}

\newcommand{\GD}{\Delta}

\newcommand{\GG}{\Gamma}

\newcommand{\GO}{\Omega}

\newcommand{\beq}{\begin{equation}}
\newcommand{\eeq}{\end{equation}}

\def\ol{\overline}

\numberwithin{equation}{section}
\numberwithin{figure}{section}

\author{Hyeonbae Kang\thanks{Department of Mathematics and Institute of Applied Mathematics, Inha University, Incheon
22212, S. Korea (hbkang@inha.ac.kr, xiaofeilee@hotmail.com).}  \and Xiaofei Li\footnotemark[2] \and Shigeru Sakaguchi\thanks{Research Center for Pure and Applied Mathematics, Graduate School of Information Sciences, Tohoku University, Sendai, 980-8579, Japan (sigersak@tohoku.ac.jp).}}

\begin{document}
\title{Existence of coated inclusions of general shape weakly neutral to multiple fields in two dimensions\thanks{\footnotesize This work was
supported by NRF grants No. 2016R1A2B4011304 and 2017R1A4A1014735, JSPS KAKENHI Grant No.  JP16K13768 and by A3 Foresight Program among China (NSF), Japan (JSPS), and Korea (NRF 2014K2A2A6000567)}}

\maketitle

\begin{abstract}
A two dimensional inclusion of core-shell structure is neutral to multiple uniform fields if and only if the core and the shell are concentric disks, provided that the conductivity of the matrix is isotropic. An inclusion is said to be neutral if upon its insertion the uniform field is not perturbed at all. In this paper we consider inclusions of core-shell structure of general shape which are weakly neutral to multiple uniform fields. An inclusion is said to be weakly neutral if the field perturbation is mild. We show, by an implicit function theorem, that if the core is a small perturbation of a disk then we can coat it by a shell so that the resulting structure becomes weakly neutral to multiple uniform fields.
\end{abstract}

\section{Introduction}

Let $D$ and $\GO$ be bounded simply connected domains in $\Rbb^2$ such that $\ol{D}\subset \GO$. We regard $(\GO,D)$ as an inclusion of the core-shell structure where the core $D$ is coated by the shell $\GO\setminus D$. We then consider the following problem, which can be viewed as a conductivity problem or an anti-plane elasticity problem:
\beq\label{main}
\begin{cases}
\nabla\cdot\Gs\nabla u=0 \quad &\mbox{in }\Rbb^2,\\
u(x)-a \cdot x=O(|x|^{-1})\quad&\mbox{as }|x|\rightarrow\infty,
\end{cases}
\eeq
where $a$ is a unit vector representing the background uniform field and $\Gs$ is a piecewise constant function defined by
\beq\label{conductivity}
\Gs=\Gs_c\chi(D)+\Gs_s\chi(\GO\setminus D)+\Gs_m\chi(\Rbb^2\setminus \GO).
\eeq
Here $\chi$ denotes the characteristic function. The function $\Gs$ represents the conductivity distribution where the conductivities $\Gs_c$ (of the core), $\Gs_s$ (shell) and $\Gs_m$ (matrix) are assumed to be isotropic (scalar).

In absence of the inclusion $(\GO,D)$ the conductivity $\Gs$ is the constant $\Gs_m$ in $\Rbb^2$. So, the equation in \eqnref{main} is just the harmonic equation ($\GD u=0$ in $\Rbb^2$) and the solution is $u(x)=a \cdot x$. Thus the field is uniform, i.e., $\nabla u =a$. Therefore, $\nabla(u(x)-a \cdot x)$ in presence of $(\GO,D)$ can be viewed as the perturbation of the uniform field by insertion of inclusion.

Upon insertion of an inclusion the field is perturbed in general. However, there are inclusions of core-shell structure such that the field is not perturbed at all.
If $D$ and $\GO$ are two concentric disks in two dimensions, say $D=\{|x|<r_i\}$ and $\GO=\{|x|<r_e\}$, and the following relation holds:
\beq\label{neutral}
(\Gs_s+\Gs_c)(\Gs_m-\Gs_s)+ \Gr^2 (\Gs_s-\Gs_c)(\Gs_m+\Gs_s)=0,
\eeq
where $\Gr=r_i/r_e$, then the solution $u$ to \eqnref{main} satisfies
\beq\label{neutral2}
u(x)-a \cdot x \equiv 0 \quad\mbox{for all } x \in \Rbb^2\setminus \GO,
\eeq
in other words, the uniform field $-\nabla(a\cdot x)$ is {\it not} perturbed at all outside $\GO$. We emphasize that \eqnref{neutral} has a solution $\Gr^2$ only when
\beq\label{neutral3}
0 < \frac{(\Gs_s+\Gs_c)(\Gs_s-\Gs_m)}{(\Gs_s-\Gs_c)(\Gs_m+\Gs_s)} <1.
\eeq

But, it is proved that if the neutrality condition \eqnref{neutral2} is fulfilled for some coated structure $(\GO, D)$, then $\GO$ and $D$ are concentric disks in two dimensions \cite{KL2d, M.S}, and concentric balls in three dimensions \cite{KLS3d}, provided that $\Gs_m$ is isotropic. It is also proved in \cite{KL2d, M.S} that they are confocal ellipses if $\Gs_m$ is anisotropic (while $\Gs_c$ and $\Gs_s$ are isotropic). It is an open problem to extend this result to three dimensions. We refer to \cite{KLS3d} for a discussion on this problem: an over-determined problem for confocal ellipsoids and a formulation of the problem in terms of Newtonian potentials.

Since the neutral inclusion cannot take an arbitrary shape, we consider a notion of neutrality weaker than \eqnref{neutral2}.
The inclusion $(\GO,D)$, the core coated by the shell, is said to be weakly neutral to multiple uniform fields if the solution $u$ to \eqnref{main} satisfies
\beq\label{weakneutral}
u(x)-a \cdot x=O(|x|^{-2})\quad\mbox{as }|x|\rightarrow\infty
\eeq
for all unit vector $a$. While \eqnref{main} requires $u(x)-a \cdot x=O(|x|^{-1})$ at $\infty$, the neutrality requires $u(x)-a \cdot x \equiv 0$ outside $\GO$. Therefore, the weak neutrality condition \eqnref{weakneutral} is in between them: unlike the neutral inclusion, the weakly neutral inclusion does perturb the uniform field, but only mildly.

The weakly neutral inclusion is also called the polarization tensor vanishing structure. In fact, the solution $u$ to \eqref{main} admits the following dipole asymptotic expansion:
\beq\label{farfield}
u(x)-a\cdot x = \frac{1}{2\pi} \frac{\la Ma, x \ra}{|x|^2} + O(|x|^{-2}), \quad |x| \to \infty,
\eeq
where $M$ is a $2 \times 2$ matrix called the polarization tensor (abbreviated by PT), which is determined by $(\GO,D)$ and conductivities $(\Gs_c, \Gs_s, \Gs_m)$ (see, for example, \cite{AmKa07Book2, mbook}). One can see easily that \eqref{weakneutral} is satisfied if and only if the PT $M$ vanishes. We emphasize that the weakly neural inclusion cannot be constructed without the shell. In fact, if the shell $\GO\setminus D$ is empty, then the PT $M$ is either positive-definite or negative-definite depending on the sign of $\Gs_m-\Gs_c$ (see \cite{AmKa07Book2}).

We now formulate the problem:

\begin{itemize}
\item[]
{\bf Weakly neutral inclusion problem}. {\it Given a domain $D$ of arbitrary shape find a shell $\GO$ so that the resulting inclusion $(\GO, D)$ is weakly neutral to multiple uniform fields, or equivalently the corresponding PT $M=M(\GO, D)$ vanishes.}
\end{itemize}

Since insertion of neutral inclusions does not perturb the outside uniform field, the effective conductivity of the assemblage filled with such inclusions of many different scales is $\Gs_m$ satisfying \eqref{neutral}. This discovery made by Hashin and Shtrikman \cite{H1, H2} has a significant implication in the theory of composites for which we refer to \cite{mbook}. Since the leading order term of the low volume expansion for the effective conductivity of the dilute composite is expressed by the PT (see \cite{AKT05} and references therein), the weakly neutral inclusion is also related to the theory of composites. Neutral inclusions are also closely related to imaging and invisibility cloaking. Neutrality \eqnref{neutral2} and weak neutrality \eqnref{weakneutral} mean that the neutral inclusions cannot be probed by the uniform fields while the weakly neutral inclusion can be vaguely seen.
The neutral inclusions are also closely related to imaging and invisibility cloaking by transformation optics. It is shown in \cite{Pendry} that perfect cloaking is achieved by transforming a punctured disk (sphere) to an annulus (the same transformation was used to show non-uniqueness of the Calder\'on's problem in \cite{GLU}). It is then shown in \cite{KO2} that if we transform a disk with a small hole instead of the punctured disk (which is to avoid singularities of the conductivity), we achieve near-cloaking instead of perfect cloaking. If we coat the small hole by another disk so that the coated structure becomes neutral, and then transform it to an annulus, then near-cloaking is dramatically enhanced \cite{AKLL}.  See also \cite{AGJKLL, AKLL2,AKLLY,KO} for further developments to Helmholtz and Maxwell's equations.

In this paper we consider weakly neutral inclusions of general shape. We mention that the study on weakly neutral inclusions is in its early stage. There is a numerical study on this problem \cite{FKL}, but not a single weakly neutral inclusion other than concentric disks is known. We don't even know how to coat ellipses to achieve weakly neutral inclusions. The purpose of this paper is to show that weakly neutral inclusions of general shape do exist. In fact, we show that if the $D$ is a small perturbation of a disk, then there is a shell $\GO$ such that the resulting inclusion $(\GO, D)$ is weakly neutral to uniform fields.

To present the main result of this paper in a precise manner, let $W^{2,\infty}(T)$ be the  collection of all functions $f$ on the unit circle $T$ such that
$$
\| f \|_{2,\infty}: = \| f \|_\infty + \| f' \|_\infty + \| f'' \|_\infty < \infty,
$$
where $T$ is parametrized as
$$
T =\left\{ (\cos\eta, \sin\eta) \in \mathbb R^2  ~|~ \eta \in [0,2\pi) \right\}.
$$
Denote by $\delta = \delta(\eta, \theta)$ the intrinsic distance between the two points $(\cos\eta, \sin\eta), (\cos\theta, \sin\theta)$ in $T$, that is,
\beq\label{intrinsic distance}
\delta(\eta, \theta) = \min\{ |\eta-\theta|, 2\pi+(\eta-\theta), 2\pi -(\eta-\theta)\}\ (\in [0,\pi]) .
\eeq
The extrinsic distance between the two points in $\mathbb R^2$ is given by
\beq\label{extrinsic distance}
|(\cos\eta, \sin\eta) - (\cos\theta, \sin\theta)| = \sqrt{2(1-\cos(\eta-\theta))}.
\eeq
These two distances are comparable as
\beq\label{relationship between distances}
\frac 2\pi \delta(\eta, \theta) \le \sqrt{2(1-\cos(\eta-\theta))}\le \delta(\eta, \theta).
\eeq

Here and throughout this paper $\| \cdot \|_p$ denotes the usual $L^p$ norm. Let $D_0: =\{|x|<r_i\}$ for some radius $r_i$, and let $D_h$ be the perturbation of $D_0$ whose boundary is given by
\beq\label{Dh}
\p D_h=\left\{ ~x ~|~x =(r_i+ h(\hat{x}))\hat{x}, \quad |\hat{x}|=1 ~ \right\}.
\eeq
The perturbation function $h$ belongs to $W^{2,\infty}(T)$.

To find the shell, we choose $r_e$ so that $r_i$ and $r_e$ satisfy \eqnref{neutral} for given conductivities $\Gs_c$, $\Gs_s$ and $\Gs_m$ satisfying \eqnref{neutral3}, and let $\GO_0:= \{|x|<r_e\}$. Then $(\GO_0, D_0)$ is neutral. For $b \in W^{2,\infty}(T)$, we define $\GO_b$ as a perturbation of $\GO_0$ by
\beq\label{GOb}
\p \GO_b =\left\{ ~x ~|~x =(r_e + b(\hat{x})) \hat{x}, \quad |\hat{x}|=1 ~ \right\}.
\eeq
However, for $\GO_b$ we restrict ourselves to the three dimensional subspace, denoted by $W_3$, spanned by $\{ 1, \cos 2\Gt, \sin 2 \Gt \}$. We then identify $b=(b_1, b_2, b_3) \in \Rbb^3$ with
$$
b(\Gt)=b(\hat{x})=b_1+b_2\cos 2\Gt + b_3\sin 2\Gt,
$$
where $\hat{x}= (\cos \Gt, \sin\Gt)$. Note that we use notation $b$ for elements of both $\Rbb^3$ and $W_3$, but it does not cause any confusion.

If $h$ and $b$ are sufficiently small, then $D_h \subset \GO_b$ and hence the PT corresponding to $(\GO_b, D_h)$, which we denote by $M=M(h,b)$, is well-defined. The follwoing is the main result of this paper:
\begin{thm}\label{main_thm}
There is $\Ge>0$ such that for each $h \in W^{2,\infty}(T)$ with $\| h \|_{2,\infty}<\Ge$ there is $b=b(h) \in \Rbb^3$ such that
\beq
M(h, b(h))=0,
\eeq
namely, the inclusion $(\GO_{b(h)}, D_h)$ of the core-shell structure is weakly neutral to multiple uniform fields. The mapping $h \mapsto b(h)$ is continuous.
\end{thm}

Theorem \ref{main_thm} shows that domains $D_h$ which are local perturbations of a disk can be coated by domains of the form $\GO_b$ so that the resulting inclusions become weakly neutral. The result is rather surprising since the shell is defined only by three bases, namely, $1$, $\cos 2\Gt$ and $\sin 2\Gt$. To the best of our knowledge, this is the first result to show existence of weakly neutral inclusions of general shape other than disks. On the other hand, this is an existence proof, and we do not know how to find $b$ to define the shell $\GO_b$. In this regards, it is worth mentioning that there is yet other way of achieving weakly neutral inclusions: by introducing an imperfect parameter on the boundary $\p D$. We refer to the very recent work \cite{KL} for construction of the imperfect parameter to achieve weakly neutral inclusions of arbitrary shape.

By switching roles of $h$ and $b$, one can prove the following theorem:
\begin{thm}\label{main_thm2}
There is $\Ge>0$ such that for each $h \in W^{2,\infty}(T)$ with $\| h \|_{2,\infty}<\Ge$ there is $b=b(h) \in \Rbb^3$ such that
the inclusion $(\GO_h, D_{b(h)})$ is weakly neutral to multiple uniform fields.
\end{thm}

Let us briefly describe how we prove Theorem \ref{main_thm}. It is known that $M(h,b)$ is a $2 \times 2$ symmetric matrix, i.e., it is of the form
\beq\label{PT}
M(h,b)=\begin{bmatrix}
	m_{11}(h,b) & m_{12}(h,b)\\
	m_{21}(h,b) & m_{22}(h,b)
\end{bmatrix}.
\eeq
By identifying $M$ with $(m_{11}, m_{22}, m_{12})$, we may regard it as a function from $U \times V$ into $\Rbb^3$, where $U$ is a small neighborhood of $0$ in $W^{2,\infty}(T)$ and $V$ is a small neighborhood of $0$ in $\Rbb^3$, i.e.,
$$
M: U \times V \subset W^{2,\infty}(T) \times \Rbb^3 \to \Rbb^3.
$$
Moreover, since $(\GO_0, D_0)$ is neutral to multiple fields, we have $M(0,0)=0$. We then show
\beq\label{bJacob}
\frac{\p (m_{11}, m_{22}, m_{12})}{\p (b_1,b_2, b_3)} (0,0) \neq 0.
\eeq
Then, Theorem \ref{main_thm} follows from the implicit function theorem in the following form \cite{KP}:
\begin{thm}[The implicit function theorem]\label{thm:ift}
Let $X$ be a Banach space. Let $U\times V$ be an open subset of $X\times \mathbb{R}^3$. Suppose that
$$
F=(F_1,F_2,F_3):(x,y)\in U\times V\rightarrow \mathbb{R}^3
$$
is continuous and has the property that the derivative of $F$ with respect to $y$ exists and is continuous at each point of $U\times V$. Further assume that at point $(x_0,y_0)\in U\times V$,
$$
F(x_0,y_0)=0 \quad \mbox{and} \quad   \frac{\p F}{\p y}(x_0,y_0) \neq 0.
$$
Then there exist neighborhood $N_1\subset U$ of $x_0$ and neighborhood $N_2\subset V$ of $y_0$ such that, for each $x$ in $N_1$, there is a unique $y\in N_2$ satisfying
$$
F(x,y)=0.
$$
The function $\hat{y}$, thereby uniquely defined near $x_0$ by the condition $\hat{y}(x)=y$, is continuous.
\end{thm}

This paper is organized as follows. In section \ref{sec:PT}, we review the definition of the PT in terms of a system of integral equations. In section \ref{sec:integral}, we consider stability of the system of integral equations. Theorem \ref{main_thm} is proved in section \ref{sec:proof}.

\section{Preliminary: layer potentials and PT}\label{sec:PT}

In this section we review basic properties of the PT and the related system of integral equations.

Let
$$
G(x)=\frac{1}{2\pi} \ln |x|,
$$
the fundamental solution to the Laplacian in two dimensions. For a bounded simple closed Lipschitz continuous curve $\GG$,
the single and double layer potentials of a function $\Gvf\in L^2(\Gamma)$ are defined by
\begin{align*}
\Sg[\Gvf](x) &:= \int_{\GG} G(x-y) \Gvf(y)\, ds(y), \quad x\in \Rbb^2, \\
\Dg[\psi](x) &:= \int_{\GG} \p_{\nu_y} G(x-y) \Gvf(y)\, ds(y), \quad x\in \Rbb^2 \setminus \GG,
\end{align*}
where $\p_{\nu_y}$ denotes the outward normal derivative with respect to $y$-variables. It is well known (see, for example, \cite{AmKa07Book2}) that the following jump relations hold:
\begin{align*}
\p_\nu \Sg[\Gvf](x)\big|_{\pm} &=(\pm \ds\frac{1}{2}I+\Kcal_{\GG}^{*})[\Gvf](x), \quad\mbox{a.e. } x\in\GG, \\
\Dg[\Gvf](x)\big|_{\pm} &=(\mp\ds\frac{1}{2}I+\Kg)[\Gvf](x), \quad\mbox{a.e. } x\in\GG,
\end{align*}
where the operator $\Kg$ on $\GG$ is defined by
\begin{equation*}
\Kg[\Gvf](x)= \frac{1}{2\pi} \int_{\GG} \frac{\la y-x,\nu(y)\ra}{|x-y|^2} \Gvf(y)\, ds(y),
\end{equation*}
and $\Kg^*$ is the $L^2$-adjoint of $\Kg$, namely,
\begin{equation*}
\Kg^*[\Gvf](x)= \frac{1}{2\pi} \int_{\GG} \frac{\la x-y,\nu(x)\ra}{|x-y|^2} \Gvf(y)\, ds(y).
\end{equation*}
Here, $\la \, , \, \ra$ denotes the scalar product in $\Rbb^2$ and subscripts $\pm$ denote the limit from outside and inside $\GG$, respectively. Set
$$L^2_0(\Gamma) = \left\{ ~\Gvf \in L^2(\Gamma) ~|~ \int_{\Gamma} \Gvf \, ds =0 ~\right\}.$$

Let $\GO$ and $D$ be two bounded simply connected planar domains such that $\ol{D} \subset \GO$, whose boundaries are assumed to be Lipschtiz continuous. We consider the problem \eqnref{main} when $a=(1,0)$ and $a=(0,1)$, whose solution is denoted by $u_1$ and $u_2$, respectively.
It is known (see, for example \cite{AmKa07Book2}) that $u_l$, $l=1,2$, can be represented as
$$
u_l(x)=x_l+\SD[\Gvf_1^{(l)}](x)+\SO[\Gvf_2^{(l)}](x),\quad x\in\Rbb^2,
$$
where $(\Gvf_1^{(l)},\Gvf_2^{(l)})\in L^2_0(\p D)\times  L^2_0(\p \GO)$ is the unique solution to the integral equations
\begin{equation}\label{pp}
\begin{bmatrix}
 -\lambda I+\KDS & \p_\nu\SO\\
 \p_\nu\SD & -\mu I+\KOS
\end{bmatrix}
\begin{bmatrix}
\Gvf_1^{(l)}\\
\Gvf_2^{(l)}
\end{bmatrix}
=-\begin{bmatrix}
\nu^l_{\p D}\\
\nu^l_{\p \GO}
\end{bmatrix},
\end{equation}
where $\nu^l_{\p D}$ is the $l$-th component of the outward unit normal vector $\nu_{\p D}$ to $\p D$, and $\nu^l_{\p\GO}$ is defined likewise. Here the numbers $\Gl$ and $\Gm$ are given by
\beq\label{lambdamu}
\Gl=\frac{\Gs_c+\Gs_s}{2(\Gs_c-\Gs_s)}\quad\mbox{and}\quad \mu=\frac{\Gs_s+\Gs_m}{2(\Gs_s-\Gs_m)}.
\eeq

We emphasize that in the expression of the operator $\p_\nu\SO$ appearing in \eqnref{pp} $\p_\nu$ denotes the normal derivative on $\p D$. Thus $\p_\nu\SO$ is a compact operator from $L^2_0(\p \GO)$ into $L^2_0(\p D)$. Likewise, $\p_\nu\SD$ is a compact operator from $L^2_0(\p D)$ into $L^2_0(\p \GO)$. So the operator in \eqnref{pp} is a compact perturbation of the operator
$$
\begin{bmatrix}
-\Gl I+\KDS & 0\\
0 & -\mu I+\KOS
\end{bmatrix}.
$$
Since $\Gs_c,\Gs_s,\Gs_m>0$, we have $|\Gl|, |\mu|>1/2$, and hence the operator above is invertible on $L^2_0(\p D)\times  L^2_0(\p \GO)$.

The PT $M=M(\GO,D)=(m_{ll'})_{l,l'=1}^2$ of the core-shell structure $(\GO,D)$ is defined by
\beq\label{PT0}
m_{ll'}=\int_{\p D}x_{l'}\Gvf_1^{(l)} \, ds+\int_{\p \GO}x_{l'} \Gvf_2^{(l)} \, ds, \quad l,l'=1,2.
\eeq
The far-field expansion \eqnref{farfield} of the solution $u$ to \eqnref{main} holds with this PT. It is known that $M$ is a symmetric matrix.

\section{Stability properties of the integral equations}\label{sec:integral}

We consider the system of integral equations \eqnref{pp} when $D=D_h$ and $\GO=\GO_b$ where $D_h$ and $\GO_b$ are defined by \eqnref{Dh} and \eqnref{GOb}, respectively:
\begin{equation}\label{pp1}
\begin{cases}
\ds \left( -\Gl I+\Kcal^*_{\p D_h} \right) [\Gvf_1] + \p_\nu\Scal_{\p\GO_b} [\Gvf_2] = \psi_1 \quad &\mbox{on } \p D_h ,\\
\ds \p_\nu\Scal_{\p D_h}[\Gvf_1] + \left( -\Gm I+\Kcal^*_{\p\GO_b} \right) [\Gvf_2] = \psi_2 \quad &\mbox{on } \p\GO_b ,
\end{cases}
\end{equation}
on $L^2_0(\p D_h) \times L^2_0(\p\GO_b)$. This system of equations admits a unique solution and there is a constant $C=C(h,b)$ such that
\beq\label{bounded}
\| \Gvf_1 \|_{L^2(\p D_h)} + \| \Gvf_2 \|_{L^2(\p\GO_b)} \le C (\| \psi_1 \|_{L^2(\p D_h)} + \| \psi_2 \|_{L^2(\p\GO_b)}).
\eeq

We now transform \eqnref{pp1} to a system of integral equations on $L^2_0(T)^2$ which is the collection of square integrable functions on the unit circle with the mean zero.
Recall that
$$
\Kcal^*_{\p D_h} [\Gvf_1](x) = \frac{1}{2\pi} \int_{\p D_h} \frac{\la x-y,\nu(x)\ra}{|x-y|^2} \Gvf_1(y)\, ds(y), \quad x \in \p D_h.
$$
Any point $x \in \p D_h$ can be written as
\beq\label{icoordi}
x =x_{i,h}(\eta):=(r_i+ h(\eta))(\cos\eta,\sin\eta), \quad \eta\in[0,2\pi).
\eeq
Then the normal vector $\nu(x_{i,h}(\eta))=:\nu_{i,h}(\eta)$ on $\p D_h$ is given by
\beq\label{inormal}
\nu_{i,h}(\eta) J_{i,h}(\eta) = \left((r_i + h(\eta))\cos\eta + h'(\eta) \sin\eta , (r_i + h(\eta)) \sin\eta - h'(\eta)\cos\eta \right) ,
\eeq
where
$$
J_{i,h}(\eta):=\sqrt{(r_i+h(\eta))^2+h'(\eta)^2}.
$$
By substituting $y= x_{i,h}(\Gt)$ we have
$$
\Kcal^*_{\p D_h} [\Gvf_1](x_{i,h}(\eta)) = \frac{1}{2\pi} \int_0^{2\pi} \frac{\la x_{i,h}(\eta) -x_{i,h}(\Gt) , \nu_{i,h}(\eta) \ra}{|x_{i,h}(\eta) -x_{i,h}(\Gt)|^2} \Gvf_1(x_{i,h}(\Gt)) J_{i,h}(\Gt)\, d\Gt.
$$
Let
$$
f_1(\Gt):= \Gvf_1(x_{i,h}(\Gt)) J_{i,h}(\Gt),
$$
and define the operator $A(h)$ on $L^2(T)$ by
$$
A(h)[f_1](\eta):= \int_{0}^{2\pi} A_h(\eta,\Gt) f_1(\Gt) \,  d\Gt,
$$
where the integral kernel $A_h (\eta,\Gt)$ is given by
\begin{align*}
A_h(\eta,\Gt) =\frac{1}{2\pi}\frac{\la x_{i,h}(\eta)-x_{i,h}(\Gt),\nu_{i,h}(\eta)  \ra}{|x_{i,h}(\eta)-x_{i,h}(\Gt)|^2} J_{i,h}(\eta).
\end{align*}
Then the following relation holds:
\beq\label{Arel}
A(h)[f_1](\eta) = J_{i,h}(\eta) \Kcal^*_{\p D_h} [\Gvf_1](x_{i,h}(\eta)).
\eeq

Similarly, points $x \in \p\GO_b$ can be written as
\beq\label{ecoordi}
x=x_{e,b}(\eta): =(r_e+ b(\eta))(\cos\eta,\sin\eta), \quad \eta\in[0,2\pi).
\eeq
Then the normal vector $\nu(x_{e,b}(\eta))=:\nu_{e,b}(\eta)$ on $\p\GO_b$ is given by
\beq\label{enormal}
\nu_{e,b}(\eta) J_{e,b}(\eta) = \left((r_e+b(\eta))\cos\eta + b'(\eta) \sin\eta , (r_e+b(\eta)) \sin\eta - b'(\eta)\cos\eta \right),
\eeq
where
$$
J_{e,b}(\eta):=\sqrt{(r_e + b(\eta))^2 + b'(\eta)^2}.
$$
Define the operator $B(b)$ on $L^2(T)$ by
$$
B(b)[f_2](\eta):= \int_{0}^{2\pi} B_b(\eta,\Gt) f_2(\Gt) \,  d\Gt,
$$
where the integral kernel $B_b (\eta,\Gt)$ is given by
\beq\label{keb_ker}
B_b (\eta,\Gt)=\frac{1}{2\pi}\frac{\la x_{e,b}(\eta)-x_{e,b}(\Gt),\nu_{e,b}(\eta)  \ra}{|x_{e,b}(\eta)-x_{e,b}(\Gt)|^2} J_{e,b}(\eta).
\eeq
Then the following relation holds:
\beq\label{Brel}
B(b)[f_2](\eta) = J_{e,b}(\eta) \Kcal^*_{\p D_h} [\Gvf_2](x_{e,b}(\eta)),
\eeq
where
$$
f_2(\Gt):= \Gvf_2(x_{e,b}(\Gt)) J_{e,b}(\Gt).
$$

One can also see that
\beq\label{Crel}
J_{i,h}(\eta) \p_\nu\Scal_{\p\GO_b} [\Gvf_2](x_{i,h}(\eta)) = C(h,b)[f_2](\eta),
\eeq
where the operator $C(h,b)$ is defined by
$$
C(h,b)[f_2](\eta) := \int_0^{2\pi} C_{h,b} (\eta,\Gt) f_2(\Gt) \ d\Gt,
$$
with
\begin{equation*}
C_{h,b} (\eta,\Gt)=\frac{1}{2\pi}\frac{\la x_{i,h}(\eta) - x_{e,b}(\Gt),  \nu_{i,h}(\eta)  \ra}{|x_{i,h}(\eta) - x_{e,b}(\Gt)|^2} J_{i,h}(\eta).
\end{equation*}

We further have
\beq\label{Drel}
J_{e,b}(\eta) \p_\nu\Scal_{\p D_h} [\Gvf_1](x_{e,b}(\eta)) = D(h,b)[f_1](\eta),
\eeq
where the operator $D(h,b)$ is defined by
$$
D(h,b)[f_1](\eta) := \int_0^{2\pi} D_{h,b} (\eta,\Gt) f_1(\Gt) \ d\Gt,
$$
with
\begin{equation*}
D_{h,b} (\eta,\Gt)=\frac{1}{2\pi}\frac{\la x_{e,b}(\eta) - x_{i,h}(\Gt),  \nu_{e,b}(\eta)  \ra}{|x_{e,b}(\eta) - x_{i,h}(\Gt)|^2} J_{e,b}(\eta).
\end{equation*}

We obtain from \eqnref{pp1} that
$$
\begin{cases}
\ds J_{i,h}(\eta) \left( -\Gl I+\Kcal^*_{\p D_h} \right) [\Gvf_1] (x_{i,h}(\eta)) + J_{i,h}(\eta) \p_\nu\Scal_{\p\GO_b} [\Gvf_2] (x_{i,h}(\eta)) = J_{i,h}(\eta) \psi_1 (x_{i,h}(\eta)) , \\
\ds J_{e,b}(\eta) \p_\nu\Scal_{\p D_h}[\Gvf_1] (x_{e,b}(\eta)) + J_{e,b}(\eta) \left( -\Gm I+\Kcal^*_{\p\GO_b} \right) [\Gvf_2] (x_{e,b}(\eta)) = J_{e,b}(\eta) \psi_2 (x_{e,b}(\eta)) .
\end{cases}
$$
It then follows from \eqnref{Arel}, \eqnref{Brel}, \eqnref{Crel}, and \eqnref{Drel} that
$$
\begin{cases}
\ds \left( -\Gl I+ A(h) \right) [f_1] + C(h,b) [f_2]  = g_1 , \\
\ds D(h,b) [f_1] + \left( -\Gm I+ B(b) \right) [f_2]  = g_2 ,
\end{cases}
$$
where
$$
g_1(\eta):= J_{i,h}(\eta) \psi_1 (x_{i,h}(\eta)) \quad\mbox{and}\quad g_2(\eta):= J_{e,b}(\eta) \psi_2 (x_{e,b}(\eta)).
$$
Let
\begin{equation*}
\Acal(h,b) :=
\begin{bmatrix}
-\Gl I+ A(h) & C(h,b) \\
D(h,b) & -\Gm I+ B(b)
\end{bmatrix}
\end{equation*}
and $f=(f_1,f_2)^\top$, $g=(g_1,g_2)^\top$. Then the above system of integral equations can be written in short as
\begin{equation*}
\Acal(h,b)f=g.
\end{equation*}
We emphasize that this equation is on $L^2_0(T)^2$ since $f_j, g_j \in L^2_0(T)$ for $j=1,2$. Furthermore, \eqnref{bounded} shows that there is a constant $C=C(h,b)$ such that
$$
\| \Acal(h,b)^{-1} g \|_{2}\le C \| g \|_{2} ,
$$
where $\| \cdot \|_{2}$ denotes the norm on $L^2(X)^2$.

In the rest of this section we consider the continuity and differentiability of the operator $\Acal(h,b)$. We first obtain the following proposition for the continuity.
\begin{prop}\label{prop:stab}
There is $\Ge >0$ such that if $h, b \in W^{2,\infty}(T)$ and $\| h \|_{2,\infty} + \| b \|_{2,\infty} \le \Ge$, then
\begin{itemize}
\item[(i)] $\Acal(h,b)$ is continuous at $(h,b)=(0,0)$ strongly, namely, there is a constant $C$ such that
\beq\label{contione}
\| (\Acal(h,b)- \Acal(0,0)) f \|_2 \le C(\| h \|_{2,\infty} + \| b \|_{2,\infty}) \| f \|_2
\eeq
for all $f \in L^2(T)^2$,
\item[(ii)] $\Acal(h,b)$ is continuous at $(h,b) \neq (0,0)$ weakly, namely, for each $f \in L^2(T)^2$
\beq\label{contitwo}
\| (\Acal(k,d)- \Acal(h,b)) f \|_2 \to 0
\eeq
as $\| k-h \|_{2,\infty} + \| d-b \|_{2,\infty} \to 0$.
\end{itemize}
\end{prop}
\begin{proof}
Let us first deal with the operator $B(b)$. Here we consider $b \in W^{2,\infty}(T)$, not necessarily in the subspace $W_3$. One can easily see from \eqnref{enormal} that
\begin{align*}
& \la x_{e,b}(\eta)-x_{e,b}(\Gt) ,\nu_{e,b}(\eta) J_{e,b}(\eta) \ra \\
& = (r_e+b(\eta))^2 - (r_e+b(\eta)) (r_e+b(\Gt)) \cos(\eta-\Gt) - b'(\eta) (r_e+b(\Gt)) \sin(\eta-\Gt).
\end{align*}
Therefore, we have
\beq\label{expnum}
\la x_{e,b}(\eta)-x_{e,b}(\Gt) ,\nu_{e,b}(\eta) J_{e,b}(\eta) \ra
= r_e^2 (1- \cos(\eta-\Gt)) (1 + R_1(b,\eta,\Gt)),
\eeq
where
\begin{align}
R_1(b;\eta, \Gt) &= \frac{r_e(b(\eta)+b(\Gt)) + b(\eta)b(\Gt) }{r_e^2}\nonumber  \\
&+ \frac{(b(\eta)-b(\Gt))^2+(r_e+b(\theta))\left\{b(\eta)+b^\prime(\eta)\sin(\theta-\eta)-b(\theta)\right\}}{ r_e^2 (1-\cos(\eta -\Gt))}.\label{Rone}
\end{align}
 In view of the intrinsic distance $\delta(\eta,\Gt)$ given by \eqref{intrinsic distance}, we may distinguish four cases.
For instance, let us consider the case where $\delta= \delta(\eta,\Gt) = 2\pi -(\eta-\theta)$. Then  taking the $2\pi$-periodicity into account yields that
\begin{align}
 b(\eta) + b'(\eta) \sin(\Gt-\eta)- b(\Gt) &= b(\eta) + b'(\eta) \sin(\delta)- b(\eta+\delta)\nonumber\\
&=  b(\eta) + b'(\eta) \delta- b(\eta+\delta) + O(\| b\|_{1,\infty} \delta^3)\nonumber\\
&= O(\| b\|_{2,\infty} \delta^2)\label{Taylor theorem is applied}.
\end{align}
The other three cases are dealt with similarly.  Since $(b(\eta)-b(\Gt))^2 = O(\| b\|_{1,\infty} \delta(\eta,\Gt)^2)$,
we infer that
$$
\left|(b(\eta)-b(\Gt))^2+(r_e+b(\theta))\left\{b(\eta)+b^\prime(\eta)\sin(\theta-\eta)-b(\theta)\right\}\right|\le C \| b \|_{2,\infty}  \delta(\eta,\Gt)^2.
$$
Hence it follows from \eqref{relationship between distances} that
\beq\label{Rone2}
\sup_{\eta, \Gt \in [0,2\pi)} |R_1(b; \eta, \Gt)| \le C \| b \|_{2,\infty}.
\eeq

One can also see that
$$
\left| x_{e,b}(\eta)-x_{e,b}(\Gt) \right|^2
= 2 \left[ r_e^2 + r_e (b(\eta)+b(\Gt)) + b(\eta)b(\Gt) \right] (1-\cos(\eta -\Gt)) + ( b(\eta)-b(\Gt))^2 .
$$
Thus we have
\beq\label{expden}
\left| x_{e,b}(\eta)-x_{e,b}(\Gt) \right|^2 = 2 r_e^2 (1-\cos(\eta -\Gt)) (1 + R_2(b;\eta, \Gt)),
\eeq
where
\beq\label{Rtwo}
R_2(b;\eta, \Gt) = \frac{r_e(b(\eta)+b(\Gt)) + b(\eta)b(\Gt) }{r_e^2}  + \frac{( b(\eta)-b(\Gt))^2}{2 r_e^2 (1-\cos(\eta -\Gt))}.
\eeq
One can easily see that
\beq\label{Rtwo2}
\sup_{\eta, \Gt \in [0,2\pi)} |R_2 (b;\eta, \Gt)| \le C \|b \|_{1,\infty} .
\eeq

If $\|b \|_{2,\infty} \le \Ge$ for a sufficiently small $\Ge$, then we have from \eqnref{keb_ker}, \eqnref{expnum} and \eqnref{expden} that
\beq\label{BRoneRtwo}
B_b (\eta,\Gt)=\frac{1}{4\pi}\left[ 1+ \frac{R_1(b;\eta,\Gt)- R_2(b;\eta,\Gt)}{1 + R_2(b;\eta, \Gt)} \right].
\eeq
Since
$$
B_b (\eta,\Gt) - B_0 (\eta,\Gt) =\frac{1}{4\pi} \frac{R_1(b;\eta,\Gt)- R_2(b;\eta,\Gt)}{1 + R_2(b;\eta, \Gt)} ,
$$
one can see immediately from \eqnref{Rone2} and \eqnref{Rtwo2} that
\beq\label{Bone}
\| (B(b)-B(0)) f_2 \|_2 \le C \|b\|_{2,\infty} \|f_2 \|_2
\eeq
for all $f_2 \in L^2(T)$.

Suppose $f_2 \in L^2(T)$. Then
$$
(B(d)-B(b)) [f_2](\eta) = \int_0^{2\pi} (B_d (\eta,\Gt) - B_b (\eta,\Gt)) f_2(\Gt) \, d\Gt.
$$
If $\| d -b \|_{2, \infty} \to 0$, then $B_d (\eta,\Gt) - B_b (\eta,\Gt) \to 0$ unless $\eta =\Gt$. Moreover, we have
$$
\left| (B_d (\eta,\Gt) - B_b (\eta,\Gt)) f_2(\Gt) \right| \le C(\| b \|_{2, \infty} + \| d \|_{2, \infty}) |f_2(\Gt)|.
$$
Thus, by Lebesgue's dominated convergence theorem, we infer that $(B(d)-B(b)) [f_2](\eta) \to 0$ for each $\eta$. We then apply Lebesgue's dominated convergence theorem to conclude that
\beq\label{Btwo}
\| (B(d)-B(b)) f_2 \|_2 \to 0
\eeq
for each fixed $f_2$.

One can show in a similar way that
\beq\label{Aone}
\| (A(h)-A(0)) f_1 \|_2 \le C \|h\|_{2,\infty} \|f_1 \|_2
\eeq
for all $f_1 \in L^2(T)$, and
\beq\label{Atwo}
\| (A(k)-A(h)) f_1 \|_2 \to 0
\eeq
as $\| k-h \|_{2, \infty} \to 0$ for each fixed $f_1$.

To handle $C(h,b)$, let
\beq\label{alpha}
\Ga(h,b;\eta,\Gt):=\la x_{i,h}(\eta) - x_{e,b}(\Gt),  \nu_{i,h}(\eta) J_{i,h}(\eta) \ra
\eeq
and
\beq\label{beta}
\Gb(h,b;\eta,\Gt):= |x_{i,h}(\eta) - x_{e,b}(\Gt)|^2
\eeq
so that
\beq\label{GaoverGb}
C_{h,b}(\eta,\Gt)= \frac{\Ga(h,b;\eta,\Gt)}{2\pi \Gb(h,b;\eta,\Gt)}.
\eeq
One can easily see that
$$
\sup_{\eta, \Gt \in [0,2\pi)} | \Ga(h,b;\eta,\Gt)- \Ga(k,d;\eta,\Gt)| \le C(\| h-k \|_{1,\infty} + \| b-d \|_{1,\infty})
$$
and
$$
\sup_{\eta, \Gt \in [0,2\pi)} | \Gb(h,b;\eta,\Gt)- \Gb(k,d;\eta,\Gt)| \le C(\| h-k \|_{\infty} + \| b-d \|_{\infty}).
$$
Furthermore, we have
\beq\label{lower}
\Gb(h,b;\eta,\Gt)= |x_{i,h}(\eta)-x_{e,b}(\Gt)|^2 \ge \frac{1}{4} (r_e-r_i)^2,
\eeq
provided that $\| h \|_\infty$ and $\| b \|_\infty$ are sufficiently small. Thus we infer that
$$
\sup_{\eta, \Gt \in [0,2\pi)} | C_{h,b}(\eta,\Gt)- C_{k,d}(\eta,\Gt)| \le C(\| h-k \|_{1,\infty} + \| b-d \|_{1,\infty}),
$$
from which we conclude that
\beq\label{Conetwo}
\| (C(h,b)-C(k,d)) f_2 \|_2 \le C (\| h-k \|_{1,\infty} + \| b-d \|_{1,\infty}) \| f_2 \|_2
\eeq
for all $f_2 \in L^2(T)$.

Similarly one can show that
\beq\label{Donetwo}
\| (D(h,b)-D(k,d)) f_1 \|_2 \le C (\| h-k \|_{1,\infty} + \| b-d \|_{1,\infty}) \| f_1 \|_2
\eeq
for all $f_1 \in L^2(T)$.
Now \eqnref{contione} follows from \eqnref{Bone}, \eqnref{Aone}, \eqnref{Conetwo} and \eqnref{Donetwo}, while \eqnref{contitwo} follows from \eqnref{Btwo}, \eqnref{Atwo}, \eqnref{Conetwo} and \eqnref{Donetwo}. This completes the proof.
\end{proof}

Since
$$
\Acal(h,b)^{-1} = \left\{I + \Acal(0,0)^{-1}(\Acal(h,b) - \Acal(0,0))\right\}^{-1} \Acal(0,0)^{-1},
$$
we obtain the following corollary as an immediate consequence of \eqnref{contione}.

\begin{cor}\label{cor}
There is $\Ge >0$ such that
$$
\| \Acal(h,b)^{-1} g \|_{2} \le C \| g \|_{2}
$$
for all $g \in L_0^2(T)^2$ for some $C$ independent of $h$ and $b$ satisfying $\|h\|_{2,\infty} + \|b\|_{2,\infty} < \Ge$.
\end{cor}

We now look into differentiability of $\Acal(h,b)$ with respect to $b$ when $b$  belongs to $W_3$, i.e., is of the form
\beq\label{b0}
b = b_1 + b_2 \cos 2\Gt + b_3 \sin 2\Gt.
\eeq
We then identify $b$ with $(b_1,b_2,b_3)$. In this case $\| b \|_{2,\infty}$ is equivalent to
$$
|b|_\infty := \max_{1 \le j \le 3} |b_j|.
$$

Let $\p_j$ denote the partial derivative with respect to $b_j$ ($j=1,2,3$). Thanks to \eqnref{lower}, one can easily see that
$$
\sup_{\eta,\Gt \in [0,2\pi)} \left| \p_j C_{h,b}(\eta,\Gt) \right|  \le C
$$
for some constant $C$.
Moreover, if $k \in W^{2,\infty}(T)$ and $d= d_1 + d_2 \cos 2\Gt + d_3 \sin 2\Gt$, then
$$
\sup_{\eta,\Gt \in [0,2\pi)} \left| \p_j C_{h,b}(\eta,\Gt) - \p_j C_{k,d}(\eta,\Gt) \right|  \le C (\| h-k \|_{1,\infty} + |b-d|_\infty).
$$
Thus we see that $\p_j  C(h,b)$ is bounded on $L^2(T)$ and
\beq\label{Cderi}
\left\| \left( \p_j  C(h,b) - \p_j  C(k,d) \right) f_2 \right \|_2 \le C (\| h-k \|_{1,\infty} + |b-d|_\infty) \| f_2 \|_2
\eeq
for all $f_2 \in L^2(T)$. Similarly one can see that $\p_j  D(h,b)$ is bounded on $L^2(T)$ and
\beq\label{Dderi}
\left\| \left( \p_j  D(h,b) - \p_j  D(k,d) \right) f_1 \right \|_2 \le C (\| h-k \|_{1,\infty} + |b-d|_\infty) \| f_1 \|_2
\eeq
for all $f_1 \in L^2(T)$.

If $b$ and $d$ take the form \eqnref{b0}, then one can see from explicit forms of $R_1$ and $R_2$ in \eqnref{Rone} and \eqnref{Rtwo} that for $l=1,2$ and $j,k=1,2,3$
\begin{align}
&\sup_{\eta,\Gt \in [0,2\pi)} |\p_j R_l(b,\eta,\Gt)| \le C, \label{derivative bounds}\\
&\sup_{\eta,\Gt \in [0,2\pi)} |R_l(b,\eta,\Gt)- R_l(d,\eta,\Gt)| \le C |b-d|_\infty, \label{Lipschitz bounds}\\
&\sup_{\eta,\Gt \in [0,2\pi)} |\p_k\p_j R_l(b,\eta,\Gt)| \le C, \label{2nd derivative bounds}\\
&\sup_{\eta,\Gt \in [0,2\pi)} |\p_j R_l(b,\eta,\Gt)-\p_j R_l(d,\eta,\Gt)| \le C |b-d|_\infty.\label{jkest}
\end{align}
In fact, \eqnref{Rone} and \eqref{Rtwo} yield the following two identities for $j=1,2,3$:
\begin{align}
&\p_j R_1(b,\eta,\Gt) = \nonumber\\
&\quad  \frac{r_e(\p_j b(\eta)+\p_j b(\theta)) + b(\theta)\p_j b(\eta) +b(\eta)\p_j b(\theta)}{r_e^2} + \frac{2(b(\eta)-b(\theta))(\p_jb(\eta)-\p_jb(\theta)) }{r_e^2(1-\cos(\eta-\theta))} +\nonumber\\
&\quad \frac{ \p_j b(\theta)[b(\eta) + b^\prime(\eta)\sin(\theta-\eta) -b(\theta)] + (r_e+ b(\theta))[\p_j b(\eta) + (\p_j b)^\prime(\eta)\sin(\theta-\eta) -\p_j b(\theta)]}{r_e^2(1-\cos(\eta-\theta))}, \label{pjR1}\\
&\p_j R_2(b,\eta,\Gt) = \nonumber\\
&\quad \frac{r_e(\p_j b(\eta)+\p_j b(\theta)) + b(\theta)\p_j b(\eta) +b(\eta)\p_j b(\theta)}{r_e^2}
+ \frac {(b(\eta)-b(\theta))(\p_j b(\eta)-\p_j b(\theta))}{r_e^2(1-\cos(\eta-\theta))}.\label{pjR2}
\end{align}
Then, it follows from the same arguments as in \eqref{Rone}-\eqref{Rone2} that \eqref{derivative bounds} holds and hence \eqref{Lipschitz bounds} does. Moreover, since $\p_j b(\cdot)$ and $\p_j d(\cdot)$ are independent of $b_k(k=1,2,3)$ and $d_k(k=1,2,3)$, respectively,  we have from \eqref{pjR1} and \eqref{pjR2} that \eqref{2nd derivative bounds} holds and hence \eqref{jkest} also does.

From \eqref{BRoneRtwo} we have, for $j=1,2,3$,
\beq\label{p1b}
\p_j B_b(\eta,\Gt) = \frac{1}{4 \pi} \frac{\p_j R_1(b;\eta,\Gt) - \p_j R_2(b;\eta,\Gt)}{1+R_2(b;\eta,\Gt)} - \frac{1}{4 \pi} \frac{ (R_1(b;\eta,\Gt) - R_2(b;\eta,\Gt)) \p_j R_2(b;\eta,\Gt)}{[1+R_2(b;\eta,\Gt)]^2}.
\eeq
It then follows from \eqnref{Rone2}, \eqnref{Rtwo2}, \eqnref{derivative bounds}, \eqnref{Lipschitz bounds}, \eqref{2nd derivative bounds} and \eqnref{jkest} that
for $j,k =1,2,3$
\begin{align}
&\sup_{\eta,\Gt \in [0,2\pi)} |\p_k\p_j B_b (\eta,\Gt) | \le C,\label{B2nd derivatives bounds}\\
&\sup_{\eta,\Gt \in [0,2\pi)} |\p_j B_b (\eta,\Gt) - \p_j B_d (\eta,\Gt) | \le C|b-d|_\infty.\label{Bderiest}
\end{align}
Thus we have
\beq\label{Bderi}
\left\| \left( \p_j  B(b) - \p_j  B(d) \right) f_2 \right \|_2 \le C |b-d|_\infty \| f_2 \|_2
\eeq
for all $f_2 \in L^2(T)$.

The following proposition is an immediate consequence of \eqnref{Cderi}, \eqnref{Dderi} and \eqnref{Bderi}.
\begin{prop}\label{prop:deri}
There is a constant $\Ge >0$ such that if $b$ is of the form \eqnref{b0} and $\| h\|_{1,\infty} + |b|_\infty<\Ge$, then $\p_j \Acal(h,b)$ is bounded on $L^2(T)^2$ for $j=1,2,3$. Moreover, there is $C>0$ such that if $d$ is of the form \eqnref{b0} and $\| k \|_{1,\infty} + |d|_\infty <\Ge$, then
\beq\label{Aderi}
\left\| \left( \p_j  \Acal(h,b) - \p_j  \Acal(k,d) \right) f \right \|_2 \le C (\| h-k \|_{1,\infty} + |b-d|_\infty) \| f \|_2
\eeq
for all $f \in L^2(T)^2$. Here $\p_j$ denotes the partial derivative with respect to $b_j$ ($j=1,2,3$).
\end{prop}

We now show that the following identities hold for any real numbers $a$ and $b$:
\begin{align}
\p_1 \Acal(0,0) \begin{bmatrix}
a e^{i\Gt}\\
b e^{i\Gt}
\end{bmatrix} &= \frac{1}{2r_i} \begin{bmatrix}  b\\ - a \end{bmatrix} e^{i\eta} , \label{p1A} \\
\p_2 \Acal(0,0) \begin{bmatrix}
a e^{i\Gt}\\
b e^{i\Gt}
\end{bmatrix}
& = \frac{1}{4 r_i}\begin{bmatrix}
3b\Gr^2\\
- a(1+2\Gr^2)
\end{bmatrix} e^{i3\eta} + \frac{1}{4 r_i} \begin{bmatrix}
b \\
- a(1-2\Gr^2)+2b\Gr
\end{bmatrix} e^{-i\eta} , \label{p2A} \\
\p_3 \Acal(0,0) \begin{bmatrix}
a e^{i\Gt}\\
b e^{i\Gt}
\end{bmatrix} &= \frac{i}{4 r_i} \begin{bmatrix}
- 3b\Gr^2\\
 a(1+2\Gr^2)\\
\end{bmatrix} e^{i3\eta} + \frac{i}{4 r_i} \begin{bmatrix}
b \\
- a(1-2\Gr^2) + 2b\Gr
\end{bmatrix} e^{-i\eta} . \label{p3A}
\end{align}

To prove above identities, we first compute the integral kernel of each component of $\p_j  \Acal(0,0)$.
Since operator $A(h)$ is independent of $b$, it is clear that
\begin{equation*}
\p_j A_h(\eta,\Gt) = 0, \quad \quad j=1,2,3.
\end{equation*}

The following identities can be derived immediately from \eqnref{Rone} and \eqnref{pjR1}:
\begin{align*}
R_1(0;\eta,\Gt) &= 0, \\
\p_1 R_1(b;\eta,\Gt)|_{b=0} &= \frac{2}{r_e}, \\
\p_2 R_1(b;\eta,\Gt)|_{b=0} &= \frac{\cos 2\eta}{r_e} + \frac{\cos 2\eta - \cos 2\Gt  \cos(\eta-\Gt) + 2\sin 2\eta \sin(\eta-\Gt)}{r_e(1-\cos(\eta-\Gt))}, \\
\p_3 R_1(b;\eta,\Gt)|_{b=0} & = \frac{\sin 2\eta}{r_e} + \frac{\sin 2\eta - \sin 2\Gt  \cos(\eta-\Gt) - 2\cos 2\eta \sin(\eta-\Gt)}{r_e(1-\cos(\eta-\Gt))}.
\end{align*}
The following identities can be derived from \eqref{Rtwo} and \eqnref{pjR2} through straight-forward computations:
\begin{align*}
R_2(0;\eta,\Gt) &=0, \\
\p_1 R_2(b;\eta,\Gt)|_{b=0} & = \frac{2}{r_e}, \\
\p_2 R_2(b;\eta,\Gt)|_{b=0} & = \frac{\cos 2\eta + \cos 2\Gt }{r_e}, \\
\p_3 R_2(b;\eta,\Gt)|_{b=0} & = \frac{\sin 2\eta + \sin 2\Gt }{r_e}.
\end{align*}
Let $\p_j B_0(\eta,\Gt):= \p_j B_b(\eta,\Gt)|_{b=0}$ for $j=1,2,3$.
Then we obtain from \eqnref{p1b} and above identities that
\beq\label{db1}
\p_1 B_0(\eta,\Gt) = 0  ,
\eeq
and
\begin{align*}
\p_2 B_0(\eta,\Gt) & = \frac{1}{4\pi} \frac{\cos 2\eta - \cos 2\Gt  + 2\sin 2\eta \sin(\eta-\Gt)}{r_e(1-\cos(\eta-\Gt))}, \\
\p_3 B_0(\eta,\Gt) & = \frac{1}{4\pi} \frac{\sin 2\eta - \sin 2\Gt  - 2\cos 2\eta \sin(\eta-\Gt)}{r_e(1-\cos(\eta-\Gt))}.
\end{align*}
Furthermore, we have
\begin{align*}
& \cos 2\eta - \cos 2\Gt  + 2\sin 2\eta \sin(\eta-\Gt) \\
& = \cos 2\eta - \cos 2\eta \cos 2(\eta-\Gt) - \sin 2\eta \sin 2(\eta-\Gt)  + 2\sin 2\eta \sin(\eta-\Gt) \\
& = 2\left[ \cos 2\eta (1+ \cos(\eta-\Gt)) + \sin 2\eta \sin(\eta-\Gt) \right](1-\cos(\eta-\Gt)).
\end{align*}
It then follows that
\beq\label{db2}
\p_2 B_0(\eta,\Gt) = \frac{1}{2\pi r_e} \left[ \cos 2\eta (1+ \cos(\eta-\Gt)) + \sin 2\eta \sin(\eta-\Gt) \right] .
\eeq
Similarly, we have
\beq\label{db3}
\p_3 B_0(\eta,\Gt)(\eta,\Gt) =\frac{1}{2\pi r_e} \left[ \sin 2\eta (1+ \cos(\eta-\Gt)) - \cos 2\eta  \sin(\eta-\Gt) \right].
\eeq
It then follows from \eqref{db1}, \eqref{db2} and \eqref{db3} that
\begin{align*}
\p_1 B(0)[e^{i\eta}] &= 0,  \\
\p_2 B(0)[e^{i\eta}] &= \frac{1}{2r_e} e^{-i\eta}, \\
\p_3 B(0)[e^{i\eta}] &= \frac{i}{ 2r_e} e^{-i\eta}.
\end{align*}

To compute $\p_j C_{h,b}(\eta,\Gt)$ at point $(h,b) = (0,0)$, we first observe that $\Ga$ and $\Gb$ given by \eqref{alpha} and \eqref{beta} take the form
$$
\Ga(0,b;\eta,\Gt) = \frac{1}{2} (r_i^2 + r_e^2 - 2 r_i r_e \cos(\eta-\Gt)) (1+ R_3),
$$
where
$$
R_3=R_3(b;\eta,\Gt)=\frac{r_i^2 -r_e^2 - 2 r_i b(\Gt) \cos(\eta-\Gt) }{r_i^2 + r_e^2 - 2 r_i r_e \cos(\eta-\Gt)},
$$
and
$$
\Gb(0,b;\eta,\Gt) = (r_i^2 + r_e^2 - 2 r_i r_e \cos(\eta-\Gt)) (1+ R_4),
$$
where
$$
R_4=R_4(b;\eta,\Gt)=\frac{b(\Gt) (b(\Gt) + 2r_e -2r_i \cos(\eta-\Gt))}{r_i^2 + r_e^2 - 2 r_i r_e \cos(\eta-\Gt)}.
$$
It then follows from \eqnref{GaoverGb} that
\beq\label{c0}
C_{0,b}(\eta,\Gt)= \frac{1}{4\pi}\left[ 1+ \frac{R_3(b;\eta,\Gt)- R_4(b;\eta,\Gt)}{1 + R_4(b;\eta, \Gt)} \right],
\eeq
and hence
\beq\label{dc}
\p_j C_{0,b}(\eta,\Gt) = \frac{1}{4 \pi} \frac{\p_j R_3(b;\eta,\Gt) - \p_j R_4(b;\eta,\Gt)}{1+R_4(b;\eta,\Gt)} - \frac{1}{4 \pi} \frac{ (R_3(b;\eta,\Gt) - R_4(b;\eta,\Gt) \p_j R_4(b;\eta,\Gt)}{[1+R_4(b;\eta,\Gt)]^2}.
\eeq

Straightforward computations yield the following:
\beq\label{R3R4}
\begin{split}
R_3(0;\eta,\Gt) &=\frac{r_i^2 -r_e^2 }{r_i^2 + r_e^2 - 2 r_i r_e \cos(\eta-\Gt)}, \\
R_4(0;\eta,\Gt) &=0,
\end{split}
\eeq
and
\begin{align*}
\p_1 R_3 (0;\eta,\Gt) &= \frac{- 2 r_i \cos(\eta-\Gt)}{r_i^2 + r_e^2 - 2 r_i r_e \cos(\eta-\Gt)}, \\
\p_2 R_3(0;\eta,\Gt) &= \frac{ -2 r_i \cos(\eta-\Gt)\cos2\Gt}{r_i^2 + r_e^2 - 2 r_i r_e \cos(\eta-\Gt)}, \\
\p_3 R_3(0;\eta,\Gt) &= \frac{ -2 r_i \cos(\eta-\Gt)\sin 2\Gt}{r_i^2 + r_e^2 - 2 r_i r_e \cos(\eta-\Gt)},
\end{align*}
and
\begin{align*}
\p_1 R_4(0;\eta,\Gt) &= \frac{2r_e -2r_i \cos(\eta-\Gt)}{r_i^2 + r_e^2 - 2 r_i r_e \cos(\eta-\Gt)}, \\
\p_2 R_4(0;\eta,\Gt) &= \frac{(2r_e -2r_i \cos(\eta-\Gt))\cos2\Gt}{r_i^2 + r_e^2 - 2 r_i r_e \cos(\eta-\Gt)}, \\
\p_3 R_4(0;\eta,\Gt) &= \frac{(2r_e -2r_i \cos(\eta-\Gt))\sin 2\Gt}{r_i^2 + r_e^2 - 2 r_i r_e \cos(\eta-\Gt)}.
\end{align*}
Plugging these terms into \eqref{dc} we have
\begin{align}
\p_1 C_{0,0}(\eta,\Gt) &= - \frac{1}{2\pi} \frac{2r_i^2 r_e - r_i (r_i^2 + r_e^2) \cos(\eta-\Gt) }{[r_i^2 + r_e^2 - 2 r_i r_e \cos(\eta-\Gt)]^2}, \label{dc1} \\
\p_2 C_{0,0}(\eta,\Gt) &= \p_1 C_{0,0}(\eta,\Gt) \cos 2\Gt , \label{dc2} \\
\p_3 C_{0,0}(\eta,\Gt) &= \p_1 C_{0,0}(\eta,\Gt) \sin 2\Gt. \label{dc3}
\end{align}

Let $\Gr = r_i/r_e$ as before, and let
$$
P_\Gr(\eta-\Gt):= \frac{1-\Gr^2}{1 - 2\Gr \cos(\eta-\Gt) + \Gr^2 } ,
$$
which is $2\pi$ times the Poisson kernel on the unit disk. It is well known that it admits the following expansion:
\beq\label{Poisson}
P_\Gr(\eta-\Gt)= \sum_{n=-\infty}^{\infty} \Gr^{|n|} e^{in(\eta-\Gt)}.
\eeq
Then we see that
\begin{align*}
\p_1 C_{0,0}(\eta,\Gt) &= - \frac{1}{2\pi} \frac{1}{r_i} \frac{2\Gr -(1+\Gr^2) \cos(\eta-\Gt)}{[1 - 2\Gr \cos(\eta-\Gt) + \Gr^2 ]^2} \\
& =  \frac{\Ga_1}{2\pi} P_\Gr(\eta-\Gt) + \frac{\Ga_2}{2\pi} P_\Gr(\eta-\Gt)^2 ,
\end{align*}
where
$$
\Ga_1:= - \frac{1+\Gr^2}{2 r_i \Gr(1-\Gr^2)}, \quad \Ga_2:= \frac{1}{2 r_i \Gr} .
$$
Further we see from \eqnref{Poisson} that
$$
P_\Gr(\eta-\Gt)^2 = \sum_{l=-\infty}^{\infty} \Gb(l) e^{il(\eta-\Gt)},
$$
where
$$
\Gb(l):= \sum_{n=-\infty}^{\infty} \Gr^{|n|+|l-n|}.
$$
Thus we have
\beq\label{p1C}
\p_1 C(0,0)[e^{i\Gt}] = \int_0^{2\pi} \p_1 C_{0,0}(\eta,\Gt) e^{i\Gt} \, d\Gt = (\Ga_1 \Gr+ \Ga_2 \Gb(1)) e^{i\eta}.
\eeq

It follows from \eqnref{dc2} that
$$
\p_2 C(0,0)[e^{i\Gt}] = \p_1 C(0,0)[\cos 2\Gt e^{i\Gt}] = \frac{1}{2} \p_1 C(0,0)[e^{i3\Gt}] + \frac{1}{2} \p_1 C(0,0)[e^{-i\Gt}].
$$
Thus we obtain from \eqnref{p1C} that
$$
\p_2 C(0,0)[e^{i\Gt}] = \frac{1}{2} (\Ga_1 \Gr^{|3|} + \Ga_2 \Gb(3)) e^{i3\eta} + \frac{1}{2} (\Ga_1 \Gr + \Ga_2 \Gb(-1)) e^{-i\eta}.
$$
We also have from \eqnref{dc3} that
$$
\p_3 C(0,0)[e^{i\Gt}] = \frac{1}{i2} (\Ga_1 \Gr^{3} + \Ga_2 \Gb(3)) e^{i3\eta} - \frac{1}{i2} (\Ga_1 \Gr + \Ga_2 \Gb(-1)) e^{-i\eta}.
$$

Note that
$$
\Gb(1)= \frac{2\Gr}{1-\Gr^2}, \quad \Gb(3)= 2\Gr^3 + \frac{2\Gr^3}{1-\Gr^2}, \quad \Gb(-1)= \frac{2\Gr}{1-\Gr^2}.
$$
Thus we have
\begin{align*}
\p_1 C(0,0)[e^{i\theta}] &= (\Ga_1 \Gr + \Ga_2 \Gb(1)) e^{i\eta} = \frac{1}{2r_i} e^{i\eta},   \\
\p_2 C(0,0)[e^{i\theta}] &= \frac{1}{2} (\Ga_1 \Gr^3 + \Ga_2 \Gb(3)) e^{i3\eta} + \frac{1}{2} (\Ga_1 \Gr + \Ga_2 \Gb(-1)) e^{-i\eta}= \frac{3\Gr^2}{4r_i} e^{i3\eta} + \frac{1}{4r_i} e^{-i\eta},   \\
\p_3 C(0,0)[e^{i\theta}] &= \frac{1}{i2} (\Ga_1 \Gr^3 + \Ga_2 \Gb(3)) e^{i3\eta} - \frac{1}{i2} (\Ga_1 \Gr + \Ga_2 \Gb(-1)) e^{-i\eta}= -i \frac{3\Gr^2}{4r_i} e^{i3\eta} + i \frac{1}{4r_i} e^{-i\eta}.
\end{align*}

To compute $\p_j D_{h,b}(\eta,\Gt)$ at $(h,b) = (0,0)$, set
$$
\xi(h,b;\eta,\Gt):=\la x_{e,b}(\eta) - x_{i,h}(\Gt),  \nu_{e,b}(\eta) J_{e,b}(\eta)  \ra,
$$
and
$$
\zeta(h,b;\eta,\Gt):= |x_{e,b}(\eta) - x_{i,h}(\Gt)|^2.
$$
Then we have
\begin{equation*}
\begin{split}
\xi(0,b;\eta,\Gt) &= (r_e + b(\eta))^2 - r_i(r_e+b(\eta))\cos(\eta-\Gt) - b'(\eta) r_i \sin(\eta-\Gt)\\
&= \frac{1}{2} (r_i^2 + r_e^2 - 2 r_i r_e \cos(\eta-\Gt)) (1+ R_5),
\end{split}
\end{equation*}
where
$$
R_5=\frac{-r_i^2 + r_e^2 + b(\eta) (4r_e +2b(\eta) - 2 r_i  \cos(\eta-\Gt)) - 2b'(\eta) r_i \cos(\eta-\Gt) }{r_i^2 + r_e^2 - 2 r_i r_e \cos(\eta-\Gt)},
$$
and
$$
\zeta(0,b;\eta,\Gt) = (r_i^2 + r_e^2 - 2 r_i r_e \cos(\eta-\Gt)) (1+R_6),
$$
where
$$
R_6 =\frac{b(\eta) (b(\eta) + 2r_e -2r_i \cos(\eta-\Gt))}{r_i^2 + r_e^2 - 2 r_i r_e \cos(\eta-\Gt)}.
$$
Then
$$
D_{0,b}(\eta,\Gt)= \frac{\xi(0,b;\eta,\Gt)}{2\pi \zeta(0,b;\eta,\Gt)} = \frac{1}{4\pi}\left[ 1+ \frac{R_5(b;\eta,\Gt)- R_6(b;\eta,\Gt)}{1 + R_6(b;\eta, \Gt)} \right]. $$
In the same way as before, one can easily get
$$
\p_1 D_{0,0}(\eta,\Gt) = \frac{1}{2\pi} \frac{2r_i^2 r_e - r_i (r_i^2 + r_e^2) \cos(\eta-\Gt) }{[r_i^2 + r_e^2 - 2 r_i r_e \cos(\eta-\Gt)]^2}.
$$
Comparing with \eqref{dc1} one can see that
$$
\p_1 D_{0,0}(\eta,\Gt) = - \p_1 C_{0,0}(\eta,\Gt).
$$
One can also see
$$
\p_2 D_{0,0}(\eta,\Gt) = - \p_1 C_{0,0}(\eta,\Gt) \cos 2\eta +E(\eta,\Gt),
$$
where
$$
E(\eta,\Gt) = \frac{1}{2\pi}  \frac{2r_i \sin 2\eta \sin(\eta-\Gt) }{r_i^2 + r_e^2 - 2 r_i r_e \cos(\eta-\Gt)},
$$
and
$$
	\p_3 D_{0,0}(\eta,\Gt) = - \p_1 C_{0,0}(\eta,\Gt) \sin 2\eta + F(\eta,\Gt),
$$
where
$$
F(\eta,\Gt) = - \frac{1}{2\pi}  \frac{2r_i \cos 2\eta \sin(\eta-\Gt) }{r_i^2 + r_e^2 - 2 r_i r_e \cos(\eta-\Gt)}.
$$

Note that
$$
E(\eta,\Gt)=
\frac{-1}{4\pi} \frac{\Gr^2}{r_i (1-\Gr^2)} (e^{i2\eta}- e^{-i2\eta}) (e^{i2(\eta-\Gt)}- e^{-i2(\eta-\Gt)}) P_\Gr(\eta-\Gt) . $$
Thus we have
$$
\int_0^{2\pi} E(\eta,\Gt) e^{i\Gt} \, d\Gt
	= - \frac{1}{2} \frac{\Gr^2}{r_i} (e^{i3\eta}-e^{-i\eta}).
$$

Similarly, we have
$$
\int_0^{2\pi} F(\eta,\Gt) e^{i\Gt} \, d\Gt
= \frac{i}{2} \frac{\Gr^2}{r_i} (e^{i3\eta} + e^{-i\eta}).
$$
Thus we have
\begin{align*}
\p_1 D(0,0)[e^{i\theta}] &=  - \p_1 C(0,0)[e^{i\theta}] = - \frac{1}{2r_i} e^{i\eta},  \\
\p_2 D(0,0)[e^{i\theta}] &= - \p_1 C(0,0)[e^{i\theta}] \cos 2\eta - \frac{1}{2} \frac{\Gr^2}{r_i} (e^{i3\eta}-e^{-i\eta}) = - \frac{1+ 2\Gr^2}{4r_i} e^{i3\eta} - \frac{1 - 2\Gr^2}{4r_i} e^{-i\eta},   \\
\p_3 D(0,0)[e^{i\theta}] &= - \p_1 C(0,0)[e^{i\theta}] \sin 2\eta + \frac{i}{2} \frac{\Gr^2}{r_i} (e^{i3\eta} + e^{-i\eta}) =  i \frac{1+ 2\Gr^2}{4r_i} e^{i3\eta} - i \frac{1 - 2\Gr^2}{4r_i} e^{-i\eta}.
\end{align*}

We then have
\begin{align}
\p_1 \Acal(0,0) \begin{bmatrix}
a e^{i\Gt}\\
b e^{i\Gt}
\end{bmatrix} &= \begin{bmatrix}
\p_1 C(0,0)[b e^{i\Gt}]\\
\p_1 D(0,0)[a e^{i\Gt}] + \p_1 B(0)[b e^{i\Gt}]
\end{bmatrix} =  \frac{1}{2r_i} \begin{bmatrix}
 b\\
- a
\end{bmatrix} e^{i\eta} ,
\end{align}
which proves \eqnref{p1A}.  \eqnref{p2A} and \eqnref{p3A} can be proved similarly.

\section{Proof of Theorem \ref{main_thm}} \label{sec:proof}

Here we prove Theorem \ref{main_thm} by showing that $m_{ll'}$ satisfies the hypothesis of Theorem \ref{thm:ift}: continuity in $(h,b)$, continuous differentiability in $b$, and \eqnref{bJacob}.

By definition \eqref{PT0} the functions $m_{l,l'}(h,b)$, $l, l'=1,2$, are given by
$$
m_{l,l'}(h,b)=\int_{\p D}x_{l'}\Gvf_1^{(l)} \, ds+\int_{\p \GO}x_{l'} \Gvf_2^{(l)} \, ds,
$$
where $\Gvf^{(l)}= (\Gvf_1^{(l)},\Gvf_2^{(l)})\in L^2_0(\p D)\times  L^2_0(\p \GO)$ is the unique solution to \eqnref{pp}.
Using changes of variables \eqnref{icoordi} and \eqnref{ecoordi}, we see that
\beq\label{m11pre}
m_{11}(h,b) =\int_0^{2\pi} (r_i+ h(\Gt))\cos\Gt \, f_{h,b,1}^{(1)}(\Gt) \, d\Gt + \int_0^{2\pi} (r_e+ b(\Gt))\cos\Gt \, f_{h,b,2}^{(1)}(\Gt) \, d\Gt ,
\eeq
where
\beq
f_{h,b,1}^{(l)}(\Gt):= \Gvf^{(l)}_1(x_{i,h}(\Gt)) J_{i,h}(\Gt), \quad f_{h,b,2}^{(l)}(\Gt):= \Gvf_2^{(l)}(x_{e,b}(\Gt)) J_{e,b}(\Gt).
\eeq

Let
\beq
p(h,b):= (r_i+ h(\Gt), r_e+ b(\Gt))^\top.
\eeq
Then, $m_{11}$ given in \eqnref{m11pre} can be rewritten as
\beq\label{m11def}
m_{11}(h,b) = \left\la \cos\Gt p(h,b), f_{h,b}^{(1)} \right\ra,
\eeq
where $f^{(l)}_{h,b}= (f^{(l)}_{h,b,1}, f^{(l)}_{h,b,2})^\top$. Here and afterwards, $\la \ , \ \ra$ denotes the inner product on $L^2(T)^2$. Likewise, we have
\beq\label{m22def}
m_{22}(h,b) = \left\la \sin\Gt p(h,b), f_{h,b}^{(2)} \right\ra,
\eeq
and
\beq\label{m12def}
m_{12}(h,b) = \left\la \sin \Gt p(h,b), f_{h,b}^{(1)} \right\ra.
\eeq

Note that $f^{(l)}_{h,b}$ is the solution of
\beq\label{inteqn}
\Acal (h,b) [f^{(l)}_{h,b}] = g^{(l)}_{h,b},
\eeq
where $g^{(l)}_{h,b}= (g^{(l)}_{h,b,1}, g^{(l)}_{h,b,2})^\top$ is given by
$$
g_{h,b,1}^{(l)}(\Gt):= - \nu^{(l)}_{\p D_h}(x_{i,h}(\Gt)) J_{i,h}(\Gt), \quad g_{h,b,2}^{(l)}(\Gt):= - \nu_{\p\GO_b}^{(l)}(x_{e,b}(\Gt)) J_{e,b}(\Gt).
$$
We see from \eqnref{inormal} and \eqnref{enormal} that $g^{(l)}_{h,b}$ is given by
\beq\label{gone}
g^{(1)}_{h,b} = - \left((r_i + h(\eta))\cos\eta + h'(\eta) \sin\eta , (r_e+b(\eta))\cos\eta + b'(\eta) \sin\eta \right)^\top
\eeq
and
\beq\label{gtwo}
g^{(2)}_{h,b} = - \left((r_i + h(\eta)) \sin\eta - h'(\eta)\cos\eta, (r_e+b(\eta)) \sin\eta - b'(\eta)\cos\eta \right)^\top .
\eeq

\medskip
\noindent{\bf Continuity in $(h,b)$}. We only prove continuity of $m_{11}$ since the others can be handled in the same way.

Suppose $k \in W^{2,\infty}(T)$ and $d \in W_3$. Then we have
$$
\Acal (k,d) [f^{(1)}_{k,d}] = g^{(1)}_{k,d}.
$$
So, we have
$$
\Acal (k,d) [f^{(1)}_{k,d} - f^{(1)}_{h,b}] = - (\Acal (k,d) - \Acal (h,b))[f^{(1)}_{h,b}] + ( g^{(1)}_{k,d} - g^{(1)}_{h,b}) .
$$
We then infer using Corollary \ref{cor} that
$$
\left\| f^{(1)}_{k,d} - f^{(1)}_{h,b} \right\|_2 \le C \left( \left\| (\Acal (k,d) - \Acal (h,b))[f^{(1)}_{h,b}] \right\|_2 + \left\| g^{(1)}_{k,d} - g^{(1)}_{h,b} \right\|_2 \right)
$$
for some constant $C$ independently of $(k,d)$ as long as $\| k \|_{2,\infty}$ and $| d|_\infty$ are sufficiently small. We then infer from \eqnref{contitwo} that
$$
\left\| (\Acal (k,d) - \Acal (h,b))[f^{(1)}_{h,b}] \right\|_2 \to 0
$$
as $\| k-h \|_{2,\infty} + |d-b|_\infty \to 0$. It is obvious from \eqnref{gone} that $\| g^{(1)}_{k,d} - g^{(1)}_{h,b} \|_2 \to 0$. Thus we have $\| f^{(1)}_{k,d} - f^{(1)}_{h,b} \|_2 \to 0$. We then conclude using \eqnref{m11def} that $m_{11}(k,d)-m_{11}(h,b) \to 0$ as $\| k-h \|_{2,\infty} + |d-b|_\infty \to 0$.

\medskip
\noindent{\bf Continuous differentiability in $b$}. By differentiating \eqnref{inteqn} with respect to $b_j$-variable, we have
$$
\Acal (h,b) [\p_j f^{(1)}_{h,b}] = \p_j g^{(1)}_{h,b}- \p_j \Acal (h,b) [f^{(1)}_{h,b}],
$$
namely,
\beq\label{derieqn}
\p_j f^{(1)}_{h,b} = \Acal (h,b)^{-1} \left[ \p_j g^{(1)}_{h,b} - \p_j \Acal (h,b) [f^{(1)}_{h,b}] \right].
\eeq
We mention that this argument is formal since we take the derivative of $f^{(1)}_{h,b}$ without proving its existence. However, this formal argument can be justified easily.

It is clear from \eqnref{gone} that $\p_j g^{(1)}_{h,b}$ is continuous in $(h,b)$. Then Corollary \ref{cor}, Proposition \ref{prop:deri} and continuity of $f^{(1)}_{h,b}$ implies that $\p_j f^{(1)}_{h,b}$ is continuous in $(h,b)$. We then obtain from \eqnref{m11def} that
$$
\p_j m_{11}(h,b) = \left\la \cos\Gt \p_j p(h,b), f_{h,b}^{(1)} \right\ra + \left\la \cos\Gt p(h,b), \p_j f_{h,b}^{(1)} \right\ra,
$$
which shows that $\p_j m_1(h,b)$ is continuous in $(h,b)$.

\medskip
\noindent{\bf Proof of \eqnref{bJacob}}. For simplicity of expression we put
$$
\psi_1(\Gt):= \cos \Gt, \quad \psi_2(\Gt):= \sin \Gt.
$$
Then derivatives of $m_{ll'}$ takes the following form
\begin{align}
\p_j m_{ll'}(0,0) &= \left\la \psi_{l'} \p_j p(0,0), f_{0,0}^{(l)} \right\ra + \left\la \psi_{l'} p(0,0), \p_j f_{0,0}^{(l)} \right\ra. \label{mderi}
\end{align}

Observe that $A_0$ and $B_0$ are constants, and hence operators $A(0)$ and $B(0)$ are trivial as operators on $L^2_0(T)$.
Thus we have
$$
\Acal(0,0) =
\begin{bmatrix}
\mu \Gr^2 I  & C(0,0) \\
D(0,0) & -\Gm I
\end{bmatrix},
$$
as an operator on $L^2_0(T)^2$. Here we used the fact $\Gl=-\mu \Gr^2$ which is a consequence of \eqref{neutral} and \eqref{lambdamu}. We see from \eqref{c0} and \eqref{R3R4} that
$$
C_{0,0}(\eta, \Gt) = \frac{1}{4 \pi} \left(1 - P_\Gr(\eta-\Gt) \right),
$$
where $P_\Gr(\eta-\Gt)$ is the Poisson kernel given in \eqnref{Poisson}. Thus we have
\beq\label{C00}
C(0,0)[e^{im\Gt}]= 	-\frac{\Gr^{|m|}}{2} e^{im\eta}, \quad m\neq 0.
\eeq
Likewise we have
\beq\label{D00}
D(0,0)[e^{im\Gt}]= \frac{\Gr^{|m|}}{2} e^{im\eta}, \quad m\neq 0.
\eeq
It then follows that $D(0,0) C(0,0) = C(0,0)D(0,0)$, and
$$
(\Gm^2 \Gr^2 I + C(0,0)D(0,0)) [e^{im\Gt}] = (\Gm^2 \Gr^2 - \frac{1}{4} \Gr^{2|m|}) e^{im\eta} .
$$
Since $|\Gm| > 1/2$ as one can see from \eqref{lambdamu}, we see that $\Gm^2 \Gr^2 I + C(0,0)D(0,0)$ is invertible on $L^2(T)$ and $\Acal(0,0)^{-1}$ is given by
\begin{equation*}
\Acal(0,0)^{-1} = -
\begin{bmatrix}
	(\Gm^2 \Gr^2 I + C(0,0)D(0,0))^{-1} &  0 \\
	0  & (\Gm^2 \Gr^2 I + C(0,0)D(0,0))^{-1}
\end{bmatrix}
\begin{bmatrix}
-\mu  I  & -C(0,0) \\
-D(0,0) & \Gm\Gr^2 I
\end{bmatrix}.
\end{equation*}
Let $(\Acal(0,0)^{-1})^*$ be the adjoint of $\Acal(0,0)^{-1}$. Then, in particular, we have
\beq\label{Ainv}
\left(\Acal(0,0)^{-1}\right)^* \begin{bmatrix} ae^{i\Gt} \\ be^{i\Gt} \end{bmatrix} = \Gg_1  e^{i\eta} \begin{bmatrix}
	-\mu a - \frac{\Gr}{2} b  \\
	\frac{\Gr}{2} a + \mu \Gr^2 b
\end{bmatrix},
\eeq
for any real constants $a$ and $b$, where
$$
\Gg_1 = \dfrac{1}{\Gr^2 (1/2+\mu)(1/2-\mu)}.
$$

We now compute the first term on the right-hand side of \eqref{mderi}. Since
$$
g^{(l)}_{0,0} = - \psi_l (r_i , r_e)^\top, \quad l=1,2,
$$
we have
\beq\label{f0}
f^{(l)}_{0,0} =  \psi_l V_1, \quad l=1,2,
\eeq
where the constant vector $V_1$ is defined by
$$
V_1 := \Gg_2 r_i\begin{bmatrix} - 1  \\ \Gr \end{bmatrix} \quad\mbox{with }
\Gg_2 = \dfrac{1}{\Gr^2 (1/2+\mu)} = \Gg_1(1/2-\mu).
$$
Since
$$\p_j p(0,0) \cdot V_1=\begin{cases}
\Gg_2 r_i \Gr& \mbox{if } j=1, \\
\Gg_2 r_i \Gr \cos 2\Gt& \mbox{if } j=2, \\
\Gg_2 r_i \Gr \sin 2\Gt& \mbox{if } j=3,
\end{cases}$$
we see that
\beq\label{1term}
\left\la \psi_{l'} \p_j p(0,0), f_{0,0}^{(l)} \right\ra =
\begin{cases}
	\Gg_2 r_i \Gr \pi & \mbox{if } (l,l',j)=(1,1,1), (2,2,1), \\
	\frac{1}{2}\Gg_2 r_i \Gr \pi & \mbox{if } (l,l',j)=(1,1,2), (1,2,3),\\
	-\frac{1}{2}\Gg_2 r_i \Gr \pi & \mbox{if } (l,l',j)=(2,2,2),  \\
	0 \quad & \mbox{otherwise}.
\end{cases}
\eeq

To compute the second term on the right-hand side of \eqref{mderi}, namely, $\la \psi_{l'} p(0,0), \p_j f_{0,0}^{(l)} \ra$, we first observe from \eqnref{derieqn} that
$$
\p_j f^{(l)}_{0,0} = \Acal (0,0)^{-1} \left[ \p_j g^{(l)}_{0,0} - \p_j \Acal (0,0) [f^{(l)}_{0,0}] \right].
$$
Thus we have
$$
\left\la \psi_{l'} p(0,0), \p_j f_{0,0}^{(l)} \right\ra = \left\la  (\Acal (0,0)^{-1})^*[\psi_{l'} p(0,0)], \p_j g^{(l)}_{0,0} - \p_j \Acal (0,0) [f^{(l)}_{0,0}] \right\ra.
$$
In view of \eqnref{Ainv}, we have
\beq
(\Acal (0,0)^{-1})^*[\psi_{l'} p(0,0)] = \psi_{l'} \Gg_1 r_i (\frac{1}{2} + \mu) \begin{bmatrix} -1 \\ \Gr \end{bmatrix} =: \psi_{l'} V_2,
\eeq
and hence
\beq\label{2term}
\left\la \psi_{l'} p(0,0), \p_j f_{0,0}^{(l)} \right\ra = \left\la \psi_{l'} V_2, \p_j g^{(l)}_{0,0} - \p_j \Acal (0,0) [f^{(l)}_{0,0}] \right\ra .
\eeq

It is convenient to use the following notation:
$$
V_3:= (0,1)^\top,
$$
and
$$
\psi_3(\Gt):=\cos 3\Gt, \quad \psi_4(\Gt):= \sin 3\Gt.
$$
Then, one can see from \eqnref{gone} and \eqnref{gtwo} that
\begin{equation*}
\begin{split}
	\p_1 g^{(1)}_{0,0} &= - \psi_1 V_3, \\
	\p_2 g^{(1)}_{0,0} &=  (1/2\psi_1 - 3/2 \psi_3) V_3 , \\
	\p_3 g^{(1)}_{0,0} &= (1/2\psi_2 - 3/2 \psi_4) V_3 ,
\end{split}
 \end{equation*}
and
\begin{equation*}
\begin{split}
	\p_1 g^{(2)}_{0,0} &= - \psi_2 V_3, \\
	\p_2 g^{(2)}_{0,0} &=  -(1/2\psi_2 + 3/2 \psi_4) V_3 , \\
	\p_3 g^{(2)}_{0,0} &= - (1/2\psi_1 + 3/2 \psi_3) V_3 .
\end{split}
 \end{equation*}
Since $V_2 \cdot V_3= \Gg_1  r_i \Gr (1/2+\mu)$, it then follows that
\beq\label{g}
\left\la \psi_{l'} V_2,  \p_j g^{(l)}_{0,0}  \right\ra =
\begin{cases}
	- \pi \Gg_1  r_i \Gr (1/2+\mu) & \mbox{if } (l,l',j)= (1,1,1), (2,2,1), \\
	\pi \Gg_1  r_i \Gr (1/2+\mu)/2 & \mbox{if } (l,l',j)= (1,1,2), (1,2,3),\\
	- \pi \Gg_1  r_i \Gr (1/2+\mu)/2 & \mbox{if } (l,l',j)=   (2,2,2),\\
	0 \quad & \mbox{otherwise}.
\end{cases}
\eeq

Let
$$
V_4: = \frac{\Gg_2}{2} (\Gr,1)^\top, \quad V_5:= \frac{\Gg_2}{4} (3 \Gr^3, 1+ 2\Gr^2)^\top .
$$
Due to \eqref{p1A}, \eqref{p2A}, \eqref{p3A} and \eqref{f0}, we have
\begin{equation*}
\begin{split}
\p_1 \Acal (0,0) [f^{(1)}_{0,0}] &= \psi_1 V_4,  \\
\p_1 \Acal (0,0) [f^{(2)}_{0,0}] &= \psi_2 V_4, \\
\p_2 \Acal (0,0) [f^{(1)}_{0,0}] &= \frac{1}{2} \psi_1 V_4 + \psi_3 V_5, \\
\p_2 \Acal (0,0) [f^{(2)}_{0,0}] &= -\frac{1}{2} \psi_2 V_4 + \psi_4 V_5, \\
\p_3 \Acal (0,0) [f^{(1)}_{0,0}] &= \frac{1}{2} \psi_2 V_4 + \psi_4 V_5, \\
\p_3 \Acal (0,0) [f^{(2)}_{0,0}] &= \frac{1}{2} \psi_2 V_4 - \psi_4 V_5.
\end{split}
\end{equation*}
Note that $V_2 \cdot V_4= 0$ and $V_2 \cdot V_5=\frac{1}{4}\Gg_1\Gg_2 r_i (\frac{1}{2} + \mu) \Gr(1-\Gr^2)$. Thus it follows that
\beq\label{f}
\left\la \psi_{l'} V_2, \p_j \Acal (0,0) [f^{(l)}_{0,0}] \right\ra =0.
\eeq

We then have from \eqref{mderi}, \eqref{1term}, \eqref{2term}, \eqref{g} and \eqref{f} that
\begin{equation*}
\p_j m_{ll'}(0,0) =
\begin{cases}
-2 \mu \tau \pi & \mbox{if } (l,l',j)= (1,1,1), (2,2,1),\\
\frac{\tau \pi}{2} & \mbox{if } (l,l',j)=  (1,1,2), (1,2,3),\\
- \frac{\tau \pi}{2} & \mbox{if } (l,l',j)=  (2,2,2),\\
0 \quad & \mbox{otherwise}.
\end{cases}
\end{equation*}
where $\tau = \dfrac{ r_e}{(1/2-\mu )(1/2+\mu )}$.

We finally obtain
\begin{equation*}
\frac{\p (m_{11}, m_{22}, m_{12})}{\p (b_1,b_2, b_3)} (0,0) = \left(\frac{\tau \pi}{2}\right)^3 \det
\begin{bmatrix}
-4\mu & 1 & 0 \\
-4\mu & -1 & 0 \\
0 & 0 & 1
\end{bmatrix} \neq 0,
\end{equation*}
which yields \eqnref{bJacob}.

\begin{rem}
By switching roles of $h$ and $b$, let $M(b,h)$ be the polarization tensor associated with domain $(\Omega_h,D_b)$. Similar computations yield
\begin{equation*}
\frac{\p (m_{11}, m_{22}, m_{12})}{\p (b_1,b_2, b_3)} (0,0) = -\left(\frac{r_e}{r_i}\frac{\tau \pi}{2}\right)^3 \det
\begin{bmatrix}
-4\mu & 1 & 0 \\
-4\mu & -1 & 0 \\
0 & 0 & 1
\end{bmatrix} \neq 0.
\end{equation*}
Then we have Theorem \ref{main_thm2}.
\end{rem}


\end{document}